\documentclass{amsart}

\newtheorem{theorem}{Theorem}[section]
\newtheorem{lemma}[theorem]{Lemma}
\newtheorem{proposition}[theorem]{Proposition}

\newtheorem{claim}{Claim}
\theoremstyle{remark}
\newtheorem{remark}[theorem]{Remark}

\newcommand{\Z}{\mathbb{Z}}
\newcommand{\Q}{\mathbb{Q}}

\title[Purely cosmetic surgeries and Jones polynomials, II]{Purely cosmetic surgeries on knots with low-span Jones polynomials, II}

\author{Kazuhiro Ichihara}
\address{Department of Mathematics, College of Humanities and Sciences, Nihon University, 3-25-40 Sakurajosui, Setagaya-ku, Tokyo 156-8550, JAPAN}
\email{ichihara.kazuhiro@nihon-u.ac.jp}

\author{Kengo Kishimoto}
\address{Department of Mathematics, Osaka Institute of Technology, 5-16-1 Omiya, Asahi-ku, Osaka 535-8585, JAPAN}
\email{kengo.kishimoto@oit.ac.jp}



\thanks{This work was supported by JSPS KAKENHI Grant Number JP26K22260.}

\date{\today}

\begin{document}

\begin{abstract}
We show that a knot in the 3-sphere with a nontrivial Jones polynomial
of span at most $14$ admits no purely cosmetic surgery.
This extends a result of the first author for span at most $11$.
We determine all Laurent polynomials of span $15$ satisfying the
conditions on the Jones polynomial used in the proof.
These polynomials form an infinite family.
We also show that a knot of braid index at most $4$ whose Jones
polynomial has odd span admits no purely cosmetic surgery.
\end{abstract}

\maketitle

\section{Introduction}

The Jones polynomial $V_L(t)\in\Z[t^{\pm1/2}]$ of an oriented link
$L$ is an isotopy invariant characterized by the following
normalization and skein relation \cite{Jones87}:
\begin{align*}
V_O(t)&=1,\\
t^{-1}V_{L_+}(t)-tV_{L_-}(t)
&=(t^{1/2}-t^{-1/2})V_{L_0}(t),
\end{align*}
where $O$ denotes the unknot and $(L_+,L_-,L_0)$ is a usual skein triple.

The \emph{span} of the Jones polynomial of a link $L$ is defined by
\[
\operatorname{span} V_L(t)
=
\max\deg V_L(t)-\min\deg V_L(t).
\]
For a knot $K$, we have $V_K(t)\in\Z[t^{\pm1}]$.

In this paper, we prove the following extension of a result of the first
author \cite{ichihara2026purelycosmeticsurgeriesknots}.

\begin{theorem}\label{thm:main}
If a knot $K$ in the 3-sphere $S^3$ has a nontrivial Jones polynomial
$V_K(t)$ and $\operatorname{span} V_K(t) \le 14$, then $K$ does not
admit purely cosmetic surgery.
\end{theorem}

The case $\operatorname{span} V_K(t) \le 11$ was proved in
\cite[Theorem 1.1]{ichihara2026purelycosmeticsurgeriesknots}.

These results support the conjecture that no nontrivial knot in $S^3$
admits purely cosmetic surgery
(see \cite[Problem 1.12(c)]{K3List}).
Here a pair of Dehn surgeries on a nontrivial knot in $S^3$ along
distinct slopes is said to be \emph{purely cosmetic} if there is an
orientation-preserving homeomorphism between the resulting 3-manifolds.

We also determine all Laurent polynomials of span $15$ satisfying the
conditions on the Jones polynomial used in the proof of
Theorem~\ref{thm:main}.
These polynomials form an infinite family
(see Section~\ref{sec:span15}).
Thus these conditions alone do not rule out span $15$.

Combining Theorem~\ref{thm:main} with known results on knots of small
braid index, we obtain the following proposition.
Here $br(K)$ denotes the braid index of a knot $K$ in $S^3$.

\begin{proposition}\label{prop:braid-index}
Let $K$ be a knot in $S^3$ with Jones polynomial $V_K(t)$.
If $br(K) \le 4$ and $\operatorname{span} V_K(t)$ is odd, then $K$
admits no purely cosmetic surgery.
\end{proposition}

In particular, if a counterexample to the conjecture above has a Jones
polynomial of span $15$, then its braid index must be at least $5$.

\section{Proof of Theorem~\ref{thm:main}}

For each integer $m\ge2$, let $\zeta_m=e^{2\pi i/m}$.
We normalize the Alexander polynomial so that
$\Delta_K(t)=\Delta_K(t^{-1})$ and $\Delta_K(1)=1$.
We first establish the following lemma, which will be used to obtain the
condition at a primitive sixth root of unity.

\begin{lemma}\label{lem:6throot}
Let $K$ be a knot in $S^3$ with Jones polynomial $V_K(t)$.
If $V_K(-1)=1$, then $V_K(\zeta_6)=1$.
\end{lemma}

We introduce some notation for the proof of Lemma~\ref{lem:6throot}.

Since $\zeta_6=e^{\pi i/3}$, we have
\[
\zeta_6^6=1,\qquad
\overline{\zeta_6}=\zeta_6^{-1}=\zeta_6^5,\qquad
\zeta_6+\overline{\zeta_6}=2\cos(\pi/3)=1.
\]
Let $\Z[\zeta_6]$ be the ring of complex numbers that can be written as
polynomials in $\zeta_6$ with integer coefficients.
Since $\zeta_6^{-1}=\zeta_6^5$, negative powers of $\zeta_6$ also lie
in $\Z[\zeta_6]$.

For $a,b,c\in\Z[\zeta_6]$ with $c\ne0$, we write
$a\equiv b\pmod{c}$ if $a-b=cx$ for some $x\in\Z[\zeta_6]$.
This relation is compatible with addition and multiplication.

Since $\zeta_6^2=\zeta_6-1$, every element of $\Z[\zeta_6]$ can be written
as $x+y\zeta_6$ with $x,y\in\Z$.
For $a\in\Z[\zeta_6]$, define its \emph{norm} by
\[
N(a)=a\bar a=|a|^2.
\]
If $a=x+y\zeta_6$, then $N(a)=x^2+xy+y^2$, so $N(a)$ is a non-negative
integer.
The norm is multiplicative, that is, $N(ab)=N(a)N(b)$.
Hence, if $a$ divides $b$ in $\Z[\zeta_6]$, then $N(a)$ divides $N(b)$
in $\Z$.

\begin{proof}[Proof of Lemma~\ref{lem:6throot}]
By Lickorish--Millett \cite[Theorem 3]{LM86}, we have
\[
V_K(\zeta_6)=\pm(i\sqrt3)^d,
\]
where $d=\dim_{\mathbb{F}_3} H_1( \Sigma_2(K) ; \mathbb{F}_3 )$ and
$\Sigma_2(K)$ is the double branched cover of $S^3$ along $K$.
On the other hand, $H_1(\Sigma_2(K);\Z)$ is a finite group of order $\vert\Delta_K(-1)\vert=\det(K)$ (see \cite{Lickorish}).
Since $V_K(-1)=\Delta_K(-1)=1$ (see \cite{Jones}), this order is $1$, and so $H_1(\Sigma_2(K);\Z)=0$.
Hence $d=0$, and $V_K(\zeta_6)=\pm1$.

To determine the sign of $V_K(\zeta_6)$, we prepare the following claim.

\begin{claim}\label{clm:congruence}
For any Laurent polynomial $f(t)\in\Z[t^{\pm1}]$, we have
\[
f(\zeta_6)\equiv f(-1) \pmod{1+\zeta_6}.
\]
\end{claim}

\begin{proof}
Since $\zeta_6-(-1)=1+\zeta_6$, we have $\zeta_6\equiv -1 \pmod{1+\zeta_6}$.
First, let $k\ge 0$.
Multiplying the congruence $\zeta_6\equiv-1$ by itself $k$ times, we have $\zeta_6^k\equiv(-1)^k$.
Next, let $k<0$.
Take an integer $j$ with $6j+k\ge0$.
Since $\zeta_6^6=1$ and $6j$ is even, we have
\[
\zeta_6^k=\zeta_6^{6j+k}\equiv(-1)^{6j+k}=(-1)^k .
\]
Thus $\zeta_6^k\equiv(-1)^k \pmod{1+\zeta_6}$ for all $k\in\Z$.
Now write $f(t)=\sum_k a_k t^k$ with $a_k\in\Z$, where the sum is finite.
Then
\[
f(\zeta_6)=\sum_k a_k\zeta_6^k\equiv\sum_k a_k(-1)^k=f(-1) \pmod{1+\zeta_6}. \qedhere
\]
\end{proof}

By Claim~\ref{clm:congruence} and the assumption $V_K(-1)=1$, we have
\[
V_K(\zeta_6)\equiv V_K(-1)=1 \pmod{1+\zeta_6}.
\]

If $V_K(\zeta_6)=-1$, then $-1\equiv1 \pmod{1+\zeta_6}$, which contradicts Claim~\ref{clm:norm} below.

\begin{claim}\label{clm:norm}
The element $1+\zeta_6$ does not divide $2$ in $\Z[\zeta_6]$.
In other words, $-1\not\equiv 1 \pmod{1+\zeta_6}$.
\end{claim}

\begin{proof}
We have
\[
N(1+\zeta_6)=(1+\zeta_6)(1+\overline{\zeta_6})=2+(\zeta_6+\overline{\zeta_6})=3, \qquad N(2)=4 .
\]
If $1+\zeta_6$ divided $2$, then $N(1+\zeta_6)=3$ would divide $N(2)=4$, which is a contradiction.
\end{proof}

Therefore $V_K(\zeta_6)=1$.
\end{proof}

\begin{remark}
The condition $d=0$ alone does not determine the sign of $V_K(\zeta_6)$.
For example, the figure-eight knot $4_1$ with $\det(4_1)=5$ and the knot $5_2$ with $\det(5_2)=7$ both have $d=0$,
but they satisfy $V_{4_1}(\zeta_6)=V_{5_2}(\zeta_6)=-1$.
These values are consistent with Claim~\ref{clm:congruence}:
we have $V_{4_1}(\zeta_6)-V_{4_1}(-1)=-1-5=-6$ and $V_{5_2}(\zeta_6)-V_{5_2}(-1)=-1-(-7)=6$.
Both are divisible by $3$, and hence by $1+\zeta_6$, since $3=(1+\zeta_6)(1+\overline{\zeta_6})$ and $1+\overline{\zeta_6}\in\Z[\zeta_6]$.
\end{remark}

We now collect the necessary conditions on the Jones polynomial that will
be used in the proof of Theorem~\ref{thm:main}.

\begin{proposition}\label{prop:conditions}
Suppose that a knot $K \subset S^3$ admits purely cosmetic surgery.
Then the Jones polynomial $V_K(t)$ of $K$ must satisfy the following conditions:
\begin{align*}
  \text{(A1)} \quad & V_K(1)      = 1, \\
  \text{(A2)} \quad & V_K'(1)     = 0, \\
  \text{(A3)} \quad & V_K''(1)    = 0, \\
  \text{(A4)} \quad & V_K'''(1)   = 0, \\
  \text{(A5)} \quad & V_K''''(1)  = 0,
\end{align*}
and also the following conditions at roots of unity:
\begin{align*}
  \text{(E2)} \quad & V_K(\zeta_2)=1, \\
  \text{(E3)} \quad & V_K(\zeta_3)=1, \\
  \text{(E4)} \quad & V_K(\zeta_4)=1, \\
  \text{(E5)} \quad & V_K(\zeta_5)=1, \\
  \text{(E6)} \quad & V_K(\zeta_6)=1.
\end{align*}
\end{proposition}

\begin{proof}
The conditions (A1)--(A5) and (E3)--(E5) follow from
\cite[Sections 2 and 3]{ichihara2026purelycosmeticsurgeriesknots}.
By \cite[Corollary 1.4]{DLME}, we have $\Delta_K(t)=1$, and hence
$V_K(-1)=\Delta_K(-1)=1$.
Thus (E2) holds, and Lemma~\ref{lem:6throot} gives (E6).
Here, since $V_K(t) \in \Z[t^{\pm1}]$, each of the conditions (E3)--(E6) does not depend on the choice of the primitive root of unity:
if $V_K(\xi)=1$ for one primitive $m$-th root of unity $\xi$, then $V_K(\xi')=1$ for every primitive $m$-th root of unity $\xi'$, since $\xi'$ is a Galois conjugate of $\xi$.
\end{proof}

\begin{remark}
The conditions (A1), (A2) and (E3) hold for every knot in $S^3$ (see \cite{Jones, Lickorish}).
\end{remark}

Now suppose that a knot $K \subset S^3$ admits purely cosmetic surgery
and that $V_K(t)\ne1$.
We will show that $\operatorname{span} V_K(t) \ge 15$. 

By (A1)--(A5), we have 
\[
(t - 1)^5 \mid (V_K (t) - 1).    
\]
Indeed, writing $V_K(t)-1 = t^{-a} g(t)$ with $g(t) \in \Z[t]$, the factor $t^{-a}$ does not vanish near $t=1$, and so the order of the zero of $V_K(t)-1$ at $t=1$ coincides with that of $g(t)$.

Also, since $V_K(t)$ is a Laurent polynomial with integer coefficients, the condition $V_K(z)=1$ at a primitive $m$-th root of unity $z$ implies 
\[
  \Phi_m (t) \mid (V_K (t) - 1), 
\]
where $\Phi_m (t)$ denotes the minimal polynomial of $z$ over $\Q$. 
Note that, since $\Phi_m(t)$ is monic, this divisibility holds in $\Z[t^{\pm 1}]$, not only in $\Q[t^{\pm 1}]$. 
Thus, by the conditions (E2)--(E6), we have
\[
  \Phi_m (t) \mid (V_K (t) - 1) \quad \text{ for $m = 2,3,4,5,6$.} 
\]
Since no two of the polynomials
$(t-1)^5,\Phi_2(t),\ldots,\Phi_6(t)$ have a common root, and
all of them are monic, we obtain
\[
Q(t):=(t-1)^5\Phi_2(t)\Phi_3(t)\Phi_4(t)\Phi_5(t)\Phi_6(t)
\mid\bigl(V_K(t)-1\bigr)
\]
in $\Z[t^{\pm1}]$.

Then $V_K(t)$ can be written as 
\[
V_K (t) = 1 + Q(t) h(t)
\]
with some Laurent polynomial $h(t) \in \Z[t^{\pm 1}]$. 
Note that the cyclotomic polynomials are written as
\begin{align*}
\Phi_2 (t) = 1 + t, \quad \Phi_3 (t) = 1 + t + t^2, \quad \Phi_4 (t) = 1 + t^2, \\
\Phi_5 (t) = 1 + t + t^2 + t^3 + t^4, \quad \Phi_6 (t) = 1 - t + t^2, 
\end{align*}
and the polynomial $Q(t)$ is calculated as
\[
\begin{aligned}
Q(t)={}&t^{16}-3t^{15}+4t^{14}-4t^{13}+3t^{12}-2t^{11}+2t^{10}-t^9\\
       &+t^7-2t^6+2t^5-3t^4+4t^3-4t^2+3t-1.
\end{aligned}
\]
Thus, $Q(t)$ has the constant term $-1$ and the highest term $t^{16}$, and so, 
\[
\operatorname{span} Q(t) = \deg Q(t) = 16
\]
holds.

To obtain $\operatorname{span} V_K(t) = \operatorname{span} (1 + Q(t) h(t)) \ge 15$, 
it suffices to show that the polynomial $h(t)$ has the non-zero terms of sufficiently high degrees. 
To verify that, we consider the power series expansion of $-1/Q(t)$. 

By direct calculation, we see that 
\[
- \frac{1}{Q(t)} = \frac{1+t}{(1-t)^2 (1-t^4) (1-t^5) (1-t^6) }. 
\]
Expand the right-hand side as a power series 
\[
- \frac{1}{Q(t)} = \sum_{ n \ge 0} a_n t^n 
\]
with integer coefficients $a_n$. 
Then $a_n > 0$ holds for any $n \ge 0$. 
Indeed, we have 
\begin{align*}
\frac{1}{(1 - t)^2} &= 1 + 2 t + 3 t^2 + 4 t^3 + \cdots \\
\frac{1}{1 - t^4} &= 1 + t^4 + t^8 + t^{12} + \cdots \\
\frac{1}{1 - t^5} &= 1 + t^5 + t^{10} + t^{15} + \cdots \\    
\frac{1}{1 - t^6} &= 1 + t^6 + t^{12} + t^{18} + \cdots ,
\end{align*}
and the numerator $1+t$ also has non-negative coefficients. 
Since all of these have non-negative coefficients and the coefficient of $t^n$ in $1/(1-t)^2$ is $n+1$, we have $a_n \ge n+1 > 0$. 
For example, direct calculations show that 
\[
-\frac{1}{Q(t)}
=
1+3t+5t^2+7t^3+10t^4+15t^5
+22t^6+30t^7+39t^8+50t^9+\cdots. 
\]

Consider the case that the degrees of the non-zero terms of $V_K(t)$ are all positive. 
Let $r \ge 1$ (resp. $s$) be the lowest (resp. highest) degree of a non-zero term of $V_K(t)$. 
Set $N := s - r = \operatorname{span} V_K(t)$. 

We first see that $h(t)$ is a polynomial, that is, $h(t) \in \Z[t]$. 
Suppose for a contradiction that $h(t)$ has a non-zero term of negative degree, and let $c\,t^{-k}$ with $k > 0$ and $c \ne 0$ be the non-zero term of $h(t)$ of the lowest degree. 
Since $Q(0) = -1$, the lowest degree term of $Q(t) h(t)$ is $-c\,t^{-k}$, and so $V_K(t) = 1 + Q(t) h(t)$ has the non-zero term $-c\,t^{-k}$ of negative degree. 
This contradicts $r \ge 1$. 

Multiplying $V_K(t)=1+Q(t)h(t)$ by
$-1/Q(t)=\sum_{n\ge0}a_nt^n$ and rearranging, we obtain
\[
h(t)=\sum_{n\ge0}a_nt^n
      -V_K(t)\sum_{n\ge0}a_nt^n
\]
as an identity of formal power series.
Since $V_K(t)$ has no terms of degree less than $r$ and the
series $\sum_{n\ge0}a_nt^n$ has only non-negative powers of $t$,
the product on the right has no terms of degree less than $r$.
Thus the coefficient of $t^n$ in $h(t)$ is $a_n$ for $0\le n<r$.
In particular, the coefficient of $t^{r-1}$ in $h(t)$ is
$a_{r-1}>0$, and hence
\[
\deg h(t)\ge r-1.
\]
On the other hand, since the highest degree of $V_K(t)$ is $s=r+N$ and $\deg Q(t) = 16$, by comparing the highest degree of $ V_K(t) = 1 + Q(t) h(t)$,  
\[
r + N = 16 + \deg h 
\]
holds, where we note that $\deg (Q(t) h(t)) = 16 + \deg h \ge 16 + r - 1 > 0$, so the constant term $1$ does not affect the highest degree term. 
Consequently, we obtain 
\[
\operatorname{span} V_K(t) = N = 16 + \deg h - r \ge 15. 
\]

The case that the degrees of the non-zero terms of $V_K(t)$ are all negative can be reduced to the above case by transformation $t \mapsto t^{-1}$. 
In fact, 
\[
t^{16} Q( t^{-1} ) = -Q(t)
\]
holds, and so, if $V_K(t) - 1 = Q(t) h(t)$, then 
\[
V_K(t^{-1}) - 1 = Q(t^{-1}) h(t^{-1}) = Q(t) \cdot \left( -t^{-16} h(t^{-1}) \right), 
\]
that is, $Q(t) \mid ( V_K(t^{-1}) -1 )$. 
Since the argument above uses only the facts that $V_K(t) - 1 = Q(t) h(t)$ for some $h(t) \in \Z[t^{\pm 1}]$ and that the degrees of the non-zero terms of $V_K(t)$ are all positive, it is applicable to $V_K(t^{-1})$. 
Thus, together with $\operatorname{span} V_K(t) = \operatorname{span} V_K(t^{-1}) $, we obtain $\operatorname{span} V_K(t) \ge 15$ in this case.

Finally we consider the case that $r \le 0 \le s$, where $r$ (resp. $s$) is the lowest (resp. highest) degree of a non-zero term of $V_K(t)$. 
Since $V_K(t) \ne 1$, we have $V_K(t)-1 = Q(t) h(t) \ne 0$. 
Then, since the lowest and the highest degrees of Laurent polynomials are additive under the multiplication of Laurent polynomials, we see that
\begin{align*}
\operatorname{span}(V_K(t) - 1) &= \operatorname{span}(Q(t)h(t)) \\
&= \operatorname{span} Q(t) + \operatorname{span} h(t) = 16 + \operatorname{span} h(t) \ge 16.   
\end{align*}
On the other hand, since $r \le 0 \le s$, we have 
\[
\operatorname{span}(V_K (t) - 1) \le \operatorname{span} V_K(t). 
\]
It follows that 
\[\operatorname{span} V_K(t) \ge \operatorname{span}(V_K(t) -1) \ge 16 .\] 

Consequently, in all the cases, if a knot $K$ in $S^3$ admits purely cosmetic surgery and $V_K(t)$ is nontrivial, then $\operatorname{span} V_K(t) \ge 15$. 
This completes the proof of Theorem~\ref{thm:main}.

\begin{remark}
The proof above actually shows that $\operatorname{span} V_K(t) \ge 16$ in the case $r \le 0 \le s$. 
Thus, if a knot $K$ admitting purely cosmetic surgery satisfies $\operatorname{span} V_K(t) = 15$, then the degrees of the non-zero terms of $V_K(t)$ are either all positive or all negative. 
\end{remark}

\section{The boundary case of span 15}\label{sec:span15}

We now describe completely the Laurent polynomials of span $15$ that satisfy the divisibility condition appearing in the proof of Theorem~\ref{thm:main}.
Retain the notation
\[
-\frac{1}{Q(t)}=\sum_{n\ge 0} a_n t^n,
\qquad a_n>0,
\]
from the preceding section.  Thus
\[
a_0,a_1,a_2,a_3,a_4,a_5,\ldots=1,3,5,7,10,15,\ldots .
\]
For each integer $r\ge 1$, define
\[
G_r(t):=1+Q(t)\sum_{j=0}^{r-1}a_jt^j.
\]
In particular, since $a_0=1$, we have $G_1(t)=1+Q(t)$.

Using $Q(t)\sum_{n\ge0}a_nt^n=-1$, we have
\[
\begin{aligned}
G_r(t)
&=1+Q(t)\sum_{j=0}^{r-1}a_jt^j\\
&=Q(t)\left(\sum_{j=0}^{r-1}a_jt^j
             -\sum_{n\ge0}a_nt^n\right)\\
&=-Q(t)\sum_{n\ge r}a_nt^n
\end{aligned}
\]
as an identity of formal power series.  Since $Q(0)=-1$, the lowest-degree term of $G_r$ is $a_rt^r$, which is non-zero.  On the other hand, from the defining finite sum and the fact that the leading coefficient of $Q$ is $1$, the highest-degree term of $G_r$ is $a_{r-1}t^{r+15}$, which is also non-zero.  Consequently
\[
\operatorname{span} G_r(t)=(r+15)-r=15,
\]
and clearly $Q\mid(G_r-1)$.  The same span and divisibility conditions hold for $G_r(t^{-1})$, by the identity $t^{16}Q(t^{-1})=-Q(t)$.
Conversely, for any $f(t)\in\Z[t^{\pm1}]$, the divisibility $Q(t)\mid(f(t)-1)$ implies the conditions (A1)--(A5) and (E2)--(E6) for $f(t)$, since $(t-1)^5\mid(f(t)-1)$ and $\Phi_m(t)\mid(f(t)-1)$ for $m=2,\dots,6$.
In particular, $G_r(t)$ and $G_r(t^{-1})$ satisfy all these conditions for every $r\ge1$.

\begin{proposition}\label{prop:span15-PCScandidate}
Suppose that a knot $K\subset S^3$ admits purely cosmetic surgery and
$\operatorname{span} V_K(t)=15$.
Then there exists a unique integer $r\ge1$ such that
\[
V_K(t)=G_r(t)\qquad\text{or}\qquad V_K(t)=G_r(t^{-1}).
\]
\end{proposition}

\begin{proof}
As shown in the proof of Theorem~\ref{thm:main}, the conditions
(A1)--(A5) and (E2)--(E6) imply $Q(t)\mid(V_K(t)-1)$.
Moreover, $V_K(t)\ne1$ since $\operatorname{span} V_K(t)=15$.
We can therefore write
\[
V_K(t)=1+Q(t)h(t),\qquad h(t)\in\Z[t^{\pm1}].
\]
As observed in the proof of Theorem~\ref{thm:main}, if $r \le 0 \le s$, where $r$ (resp. $s$) is the lowest (resp. highest) degree of a non-zero term of $V_K(t)$, then
\[
\operatorname{span} V_K(t)\ge16.
\]
Hence, under the assumption $\operatorname{span} V_K(t)=15$, the degrees of the non-zero terms of $V_K(t)$ are either all positive or all negative.

Suppose first that all non-zero terms have positive degrees.  In this case, we have $r\ge1$.
Since $\operatorname{span} V_K(t)=15$, the highest degree is $r+15$.
As in the proof of Theorem~\ref{thm:main}, we have $h(t)\in\Z[t]$.
Multiplying $V_K(t)=1+Q(t)h(t)$ by
$-1/Q(t)=\sum_{n\ge0}a_nt^n$ and rearranging, we obtain
\[
h(t)=\sum_{n\ge0}a_nt^n
      -V_K(t)\sum_{n\ge0}a_nt^n
\]
as an identity of formal power series.
Since $V_K(t)$ has no terms of degree less than $r$ and the
series $\sum_{n\ge0}a_nt^n$ has only non-negative powers of $t$,
the product on the right has no terms of degree less than $r$.
Thus the coefficient of $t^n$ in $h(t)$ is $a_n$ for $0\le n<r$.
Moreover, comparison of the highest degrees in $V_K(t)=1+Q(t)h(t)$ gives
\[
r+15=16+\deg h,
\]
and therefore
\[
\deg h=r-1.
\]
Thus $h(t)$ has no terms of degree greater than $r-1$.
Together with the coefficient comparison above, this gives
\[
h(t)=\sum_{j=0}^{r-1}a_jt^j.
\]
Thus $V_K(t)=G_r(t)$.

If all non-zero terms of $V_K(t)$ have negative degrees, then $V_K(t^{-1})$ has only positive-degree terms and still satisfies
\[
Q(t)\mid(V_K(t^{-1})-1),\qquad \operatorname{span} V_K(t^{-1})=15,
\]
by the identity $t^{16}Q(t^{-1})=-Q(t)$.  The lowest degree of a non-zero term of $V_K(t^{-1})$ is $-s\ge1$.
Applying the preceding case to $V_K(t^{-1})$, we obtain
$V_K(t)=G_{-s}(t^{-1})$.

Finally, for each $r\ge1$, the lowest degree of a non-zero term of
$G_r(t)$ is $r$, and the highest degree of a non-zero term of
$G_r(t^{-1})$ is $-r$.
Thus the integer in the statement is uniquely determined by $V_K(t)$,
since all its non-zero terms have either positive or negative degrees.
\end{proof}

The same argument applies to any Laurent polynomial
$f(t)\in\Z[t^{\pm1}]$ with $Q(t)\mid(f(t)-1)$ and
$\operatorname{span} f(t)=15$.
Thus every such polynomial is $G_r(t)$ or $G_r(t^{-1})$
for a unique integer $r\ge1$.

\begin{remark}\label{rem:census19}
Using Burton's Regina census \cite{BurtonKnots}, we performed a computer-assisted search for prime knots through $19$ crossings.
No prime knot in this range has Jones polynomial
\[
V_K(t)=G_r(t)\qquad\text{or}\qquad V_K(t)=G_r(t^{-1})
\]
for any $r\ge1$.
Note that the purely cosmetic surgery conjecture itself has been verified for all knots with at most $19$ crossings by Futer, Purcell and Schleimer \cite[Theorem 2.11]{FPS}.
The search above concerns a different question, namely, whether the polynomials $G_r(t)$ and $G_r(t^{-1})$ are realized as Jones polynomials of knots, regardless of purely cosmetic surgery.
\end{remark}

The calculations above show that the conditions
(A1)--(A5) and (E2)--(E6) do not exclude span $15$.
Indeed, the polynomials $G_r(t)$ and $G_r(t^{-1})$, $r\ge1$,
form an infinite family satisfying all these conditions.
Whether any of these polynomials is realized as the Jones polynomial
of a knot is a separate question.

\section{A braid index restriction}\label{sec:braid-index}

In this section, we prove Proposition~\ref{prop:braid-index}.
It shows that, for knots of braid index at most $4$, the oddness of the span of the Jones polynomial gives an obstruction to purely cosmetic surgery.
In particular, since the polynomials $G_r(t)$ and $G_r(t^{-1})$ given in Section~\ref{sec:span15} have span $15$, a knot admitting purely cosmetic surgery whose Jones polynomial is one of them must have braid index at least $5$.

    
\begin{proof}[Proof of Proposition~\ref{prop:braid-index}]
Suppose, to the contrary, that $K$ admits purely cosmetic surgery.
Note that, since $\operatorname{span} V_K(t)$ is odd, we have $V_K(t)\ne1$, and so $K$ is nontrivial.
Thus, by Theorem~\ref{thm:main}, we have $\operatorname{span} V_K(t)\ge15$.
Also, by Daemi--Lidman--Miller Eismeier \cite[Corollary 1.4]{DLME}, a nontrivial knot admitting purely cosmetic surgery has
\[
g(K)=2,\qquad \Delta_K(t)=1.
\]

Since $K$ is nontrivial, if $br(K)\le2$, then $K$ is a $(2,p)$-torus knot, and Moser's classification of surgeries on torus knots implies that $K$ admits no purely cosmetic surgery \cite{Moser}.
By Varvarezos \cite{Varvarezos}, a nontrivial knot which is the closure of a $3$-braid does not admit purely cosmetic surgery.
Hence we have $br(K)=4$, that is, $K$ is the closure of a $4$-braid $\beta$.
Let $e=e(\beta)$ be its exponent sum.

Regard the closure of $\beta$ as a transverse knot in the standard tight contact structure on $S^3$.
Its self-linking number is
\[
\operatorname{sl}(\widehat{\beta})=e-4.
\]
By Bennequin's inequality \cite{Bennequin}, applied using a minimal-genus Seifert surface for $K$, we have
\[
e-4=\operatorname{sl}(\widehat{\beta})
\le 2g(K)-1=3.
\]
Let $\beta^*$ be the $4$-braid obtained from $\beta$ by replacing each generator $\sigma_i$ by $\sigma_i^{-1}$.
Then $\widehat{\beta^*}$ is the mirror image of $K$ and $e(\beta^*)=-e$.
Since the Seifert genus is invariant under taking mirror images,
\[
g(\widehat{\beta^*})=g(K)=2.
\]
Applying Bennequin's inequality to the transverse closure of $\beta^*$ gives
\[
-e-4=\operatorname{sl}(\widehat{\beta^*})
\le 2g(\widehat{\beta^*})-1=3.
\]
Consequently,
\[
|e|\le7.
\]

For a knot which is the closure of a $4$-braid of exponent sum $e$, the Jones--Alexander relation for closed $4$-braids (\cite[Theorem 5]{Stoimenow4braid}, \cite[Proposition 11.11]{Jones87}) gives 
\begin{equation}\label{eq:four-braid-relation}
\begin{aligned}
&t^{-e}V_K(t)+t^eV_K(t^{-1})\\
&\quad=
\left(t^{-3/2}+t^{-1/2}+t^{1/2}+t^{3/2}\right)
\left(t^{-e/2}+t^{e/2}\right)
-
\left(t^{-2}+t^{-1}+2+t+t^2\right)\Delta_K(t).
\end{aligned}
\end{equation}
Since $\Delta_K(t)=1$ and $|e|\le7$, the degrees occurring in the first product on the right-hand side of \eqref{eq:four-braid-relation} lie in $[-5,5]$, while those occurring in the second term lie in $[-2,2]$.
Hence all non-zero terms on the right-hand side have degrees in $[-5,5]$.

Let $r$ (resp. $s$) be the lowest (resp. highest) degree of a non-zero term of $V_K(t)$, and set $N:=s-r=\operatorname{span} V_K(t)$.
Then the lowest and the highest degrees of the non-zero terms of the two summands on the left-hand side of \eqref{eq:four-braid-relation} are
\[
r-e,\ s-e
\qquad\text{and}\qquad
e-s,\ e-r,
\]
respectively.
In particular, the highest degrees of the two summands are $s-e$ and $e-r$, whose sum is $N$.  Thus
\[
M:=\max\{s-e,\,e-r\}\ge \frac{N}{2}\ge\frac{15}{2}.
\]
Since $M$ is an integer, we have $M\ge8$.
Moreover, the two highest degrees cannot be equal, since their sum $N$ is odd.
Hence the term of degree $M$ on the left-hand side comes from exactly one of the two summands and has a non-zero coefficient.
This contradicts the fact that all non-zero terms on the right-hand side have degrees in $[-5,5]$.
Therefore $K$ admits no purely cosmetic surgery.
\end{proof}

\section*{Statement on the use of AI}
AI tools (Claude, ChatGPT) were used for checking computations and arguments and for improving the exposition. 
The authors take full final responsibility for the content of this paper.

\bibliographystyle{amsplain}
\bibliography{PCSjones}

\end{document}